\documentclass[12pt]{amsart}

\usepackage[utf8]{inputenc}
\usepackage[a4paper,nomarginpar]{geometry}

\usepackage[english]{babel}
\usepackage{enumitem}
\usepackage{amsmath,amsfonts,amssymb,amsthm,amscd}
\usepackage{tikz}
\usepackage{xcolor}
\usepackage{graphics}

\allowdisplaybreaks

\makeatletter
\def\subsection{\@startsection{subsection}{2}%
        \z@{.5\linespacing\@plus.7\linespacing}{.3\linespacing}%
        {\normalfont\bfseries}}
\makeatother

\theoremstyle{theorem}
\newtheorem{theorem}{Theorem}[section]

\newtheorem{corollary}[theorem]{Corollary}

\theoremstyle{definition}

\newtheorem{example}[theorem]{Example}

\theoremstyle{remark} \theoremstyle{question} \theoremstyle{example}

\def\leaderfill{\leaders\hbox to .8em{\hss .\hss}\hfill}
\def\_#1{{\lower 0.7ex\hbox{}}_{#1}}

\def\vr{{\varphi}}

\def\Hom{\operatorname{{Hom}}}

\def\Ker{\operatorname{{Ker}}}
\def\mI{\operatorname{{Im}}}

\begin{document}


\title[On exact diagrams and strict morphisms]
      {On exact diagrams and strict morphisms}




\author{Dinam\'erico P. Pombo Jr.}
\address{Instituto de Matem\'atica e Estat\'istica, Universidade Federal
Fluminense, Rua Professor Marcos Waldemar de Freitas Reis,
s/n${}^{\underline{\rm o}}$, Bloco G, Campus do Gragoat\'a,
24210-201, Niter\'oi, RJ, Brasil.}

\curraddr{}

\email{dpombojr@gmail.com}




\date{}

\dedicatory{}

\begin{abstract}
Necessary and sufficient conditions for the exactness (in the
algebraic sense) of certain sequences of continuous group
homomorphisms are established.
\end{abstract}

\maketitle


\section{Introduction}

Two fundamental theorems of Linear Algebra
\cite[p.\ 57 and p.\ 59]{Bourbaki1},
which can be extended to arbitrary abelian categories
\cite[p.\ 128]{Grothendieck1},
assert that the exactness of certain sequences of linear mappings between
modules is equivalent to the exactness of sequences of group homomorphisms
between the corresponding abelian groups of linear mappings (the concept
of an exact sequence is due to W. Hurewicz \cite{Hurewicz}). In this work
we prove the validity of analogous results in the context of topological
groups. More precisely, we show that the notion of a strict morphism is the key ingredient which allows us
to establish the equivalence between the exactness (in the algebraic sense)
of certain sequences of continuous group homomorphisms and the exactness of
diagrams of mappings between the corresponding sets of continuous group
homomorphisms. Moreover, two applications to the theory of groups and two
applications to the theory of topological groups are provided. Similar
results may be found, for example, in \cite[II, p.\ 50]{Bourbaki2} and
\cite[p.\ 23]{Palamodov}. It should also be mentioned that in this paper,
whose main purpose is the discussion of the interplay between the classical
notions of strictness and exactness, all the proofs are of an elementary
nature.


\section{Preliminaries}

In this work the identity element of any group and a group reduced to its
identity element will be denoted by $e$. For arbitrary groups $A$ and $B$,
the group homomorphism $x \in A \mapsto e \in B$ will also be denoted by
$e$\,; and, if $u\colon A \to B$ is an arbitrary group homomorphism, its
kernel (resp.\ image) will be represented by $\Ker(u)$ (resp. $\mI(u)$).
If $A$ is an arbitrary topological group, the continuous group homomorphism
$x \in A \mapsto x \in A$ will be represented by $1_A$\,. A sequence
$$
\cdots \longrightarrow A_n \overset{u_n}{\longrightarrow}
  A_{n+1} \overset{u_{n+1}}{\longrightarrow} A_{n+2} \longrightarrow \cdots
$$
of group homomorphisms is \textit{exact\/} \cite[p.\ 15]{Lang} if
$\mI(u_n) = \Ker(u_{n+1})$ for all $n$.

In what follows \textit{morphism\/} will mean continuous group homomorphism.
The set of all morphisms from a topological group $A$ into a topological
group $B$ will be denoted by $\Hom(A,B)$; $u \in \Hom(A,B)$ is a
\textit{strict morphism\/} \cite[p.\ 30]{Bourbaki0} if the bijective morphism
$\bar{u}\colon x\Ker(u) \in A/\Ker(u) \mapsto u(x) \in \mI(u)$ is a
topological group isomorphism ($A/\Ker(u)$ endowed with the quotient
topology and $\mI(u)$ endowed with the topology induced by that of $B$).
For $u \in \Hom(A,B)$ to be a strict morphism it is necessary and sufficient
that $u$ transforms open subsets of $A$ onto open subsets of $\mI(u)$.
For $u \in \Hom(A,B)$ and a topological group $D$, the mapping
$$
\vr \in \Hom(D,A) \mapsto u\circ\vr \in \Hom(D,B)\,
({\rm resp.}\, \psi \in \Hom(B,D) \mapsto \psi \circ u \in \Hom(A,D))
$$
will be denoted by $u^*$ (resp. $u_*$).

Let us recall \cite[p.\ 301]{Grothendieck2} that, if $E$, $E'$, $E''$ are sets,
a diagram
$$
E \overset{u}{\longrightarrow} E'^{\quad\overset{v_1}{\longrightarrow}}_{\quad\,\underset{v_2}{\longrightarrow}}\quad E''
$$
of mappings is {\it exact\/} if $u$ is a bijection from $E$ onto the set
$\Ker(v_1,v_2) := \{x' \in E';\, v_1(x') = v_2(x')\}$. In particular,
if $A$, $B$, $C$, $D$ are topological groups, $u \in \Hom(A,B)$,
$v \in \Hom(B,C)$, to say that the diagram

\noindent $\Hom(D,A) \overset{u^*}{\longrightarrow} \Hom(D,B)^{\quad\overset{v^*}{\longrightarrow}}_{\quad\,\underset{\bf 1}{\longrightarrow}} \quad \Hom(D,C)$ (where ${\bf 1}(\psi)=e$ for $\psi \in \Hom(D,B)$)

\noindent({\rm resp.} $\Hom(C,D) \overset{v_*}{\longrightarrow} \Hom(B,D)^{\quad\overset{u_*}{\longrightarrow}}_{\quad\,\underset{\bf 1}{\longrightarrow}} \quad \Hom(A,D)$ (where ${\bf 1}(\psi)=e$ for $\psi \in \Hom(B,D)$))

\noindent of mappings is exact is equivalent to saying that $u^*$ is a
bijection from $\Hom(D,A)$ onto $\Ker(v^*,{\bf 1})$
(resp. $v_*$ is a bijection from $\Hom(C,D)$ onto $\Ker(u_*,\bf 1)$).

If $A$, $B$ are abelian topological groups, $\Hom(A,B)$ is an abelian group
and $u^*$, $v^*$, $u_*$, $v_*$ are group homomorphisms. In this context,
the exactness of the diagram
$\Hom(D,A) \overset{u^*}{\longrightarrow} \Hom(D,B)^{\overset{v^*}{\longrightarrow}}_{\underset{\bf 1}{\longrightarrow}} \,\, \Hom(D,C)$ \,
(resp. $\Hom(C,D) \overset{v_*}{\longrightarrow}\Hom(B,D)^{\overset{u_*}{\longrightarrow}}_{\underset{\bf 1}{\longrightarrow}} \,\,\Hom(A,D))$
of mappings is equivalent to the exactness of the sequence
$$
e \longrightarrow \Hom(D,A)\overset{u^*}{\longrightarrow} \Hom(D,B)
\overset{v^*}{\longrightarrow} \Hom(D,C)
$$
$$
({\rm resp.}\, e \longrightarrow \Hom(C,D) \overset{v_*}{\longrightarrow}\Hom(B,D)\overset{u_*}{\longrightarrow}\Hom(A,D))
$$
of group homomorphisms.


\section{The results}

In contrast with what occurs in Linear Algebra one may find abelian
topological groups $A,B,u \in \Hom(A,B)$ and $v \in \Hom(B,B)$,
in such a way that the sequence
$$
e \longrightarrow A \overset{u}{\longrightarrow} B \overset{v}{\longrightarrow} B
$$
of group homomorphisms is exact, but the sequence
$$
e \longrightarrow \Hom(B,A) \overset{u^*}{\longrightarrow} \Hom(B,B) \overset{v^*}{\longrightarrow} \Hom(B,B)
$$
of group homomorphisms is not exact. In fact, let $A$ be the additive group
of real numbers endowed with the discrete topology and $B$ the additive group
of real numbers endowed with the usual topology. Let $u \in \Hom(A,B)$ be
given by $u(x) = x$ and let $v \in \Hom(B,B)$ be given by $v(x) = e$.
It is clear that $u$ is not a strict morphism and that the sequence
$$
e \longrightarrow A \overset{u}{\longrightarrow} B \overset{v}{\longrightarrow} B
$$
of group homomorphisms is exact. Nevertheless, the sequence
$$
\Hom(B,A) \overset{u^*}{\longrightarrow} \Hom(B,B) \overset{v^*}{\longrightarrow} \Hom(B,B)
$$
of group homomorphisms is not exact. For, if it were, since
$1_B \in \Ker(v^*) = \mI(u^*)$, there would exist a $\vr \in \Hom(B,A)$
so that $u^*(\vr) = u \circ \vr = 1_B$\,; but this would imply $\vr(x) = x$
for $x \in B$, which is not continuous as a mapping from $B$ into $A$.

The above-mentioned example shows that the strictness of $u$ is essential
for the validity of the implication (a) $\Rightarrow$ (b) in the following

\begin{theorem}\label{Theorem3.1}
Let $A$, $B$, $C$ be topological groups, $u \in \Hom(A,B)$ and
$v \in \Hom(B,C)$. Then the following conditions are equivalent:

\begin{itemize}

\item[\rm (a)] $u$ is a strict morphism and the sequence
  $$
  e \longrightarrow A \overset{u}{\longrightarrow} B \overset{v}{\longrightarrow} C
  $$
  of group homomorphisms is exact;

\item[\rm (b)] for each topological group $D$, the diagram
  $$
  \Hom(D,A) \overset{u^*}{\longrightarrow} \Hom(D,B)^{\,\,\overset{v^*}{\longrightarrow}}_{\,\,\,\underset{\bf 1}{\longrightarrow}}\Hom(D,C)
  $$
  of mappings is exact.

\end{itemize}
\end{theorem}

\begin{proof}
(a) $\Rightarrow$ (b):\, The exactness of the sequence
$$
e \longrightarrow A \overset{u}{\longrightarrow} B
$$
is equivalent to the injectivity of $u$.

Let $D$ be an arbitrary topological group, and let us show the exactness
of the diagram
$$
\Hom(D,A) \overset{u^*}{\longrightarrow} \Hom(D,B)^{\,\,\overset{v^*}{\longrightarrow}}_{\,\,\,\underset{\bf 1}{\longrightarrow}}\Hom(D,C).
$$

It is obvious that the injectivity of $u^*$ follows from the injectivity
of $u$. Now let us prove that ${\rm Im}(u^*) = \Ker(v^*,{\bf 1})$.
In fact, since $\mI(u) \subset \Ker(v)$, it follows that
$\mI(u^*) \subset \Ker(v^*,{\bf 1})$. On the other hand,
if $\psi \in \Ker(v^*,{\bf 1})$, $v \circ \psi = e$\,, and hence
$\mI(\psi) \subset \Ker(v) = \mI(u)$. Consequently, there is a unique
mapping $\vr\colon D \to A$ such that $\psi = u \circ \vr$, and it is clear
that $\vr$ is a group homomorphism. Since $\psi$ is continuous and the
group isomorphism $u(x) \in \mI(u) \mapsto x \in A$ is continuous
(because $u$ is a strict morphism), we conclude that $\vr \in \Hom(D,A)$;
and, by construction, $\psi = u^*(\vr) \in \mI(u^*)$. Therefore
$\Ker(v^*,{\bf 1}) \subset \mI(u^*)$, and the equality
$\mI(u^*) = \Ker(v^*,{\bf 1})$ is established.

\smallskip

\noindent\, (b) $\Rightarrow$ (a):\, By taking $D = \Ker(u)$ endowed with
the topology induced by that of $A$ and $\vr\colon \Ker(u) \to A$ the
inclusion mapping, the hypothesis guarantees the injectivity of the mapping
$$
\Hom(\Ker(u),A) \overset{u^*}{\longrightarrow} \Hom(\Ker(u),B).
$$
Since $u^*(\varphi)=u^*(e)$, it follows that $\vr = e$\,, which implies the
exactness of the sequence
$$
e \longrightarrow A \overset{u}{\longrightarrow} B.
$$

Now, by taking $D = A$, the exactness of the diagram
$$
\Hom(A,A) \overset{u^*}{\longrightarrow} \Hom(A,B)^{\overset{v^*}{\longrightarrow}}_{\underset{\bf 1}{\longrightarrow}} \Hom(A,C)
$$
and the equality $(v^* \circ u^*)(1_A) = v \circ u$ imply
$\mI(u) \subset \Ker(v)$. And, by taking $D = \Ker(v)$ endowed with the
topology induced by that of $B$ and $\psi\colon \Ker(v) \to B$ the inclusion
mapping, one has $\psi \in \Ker(v^*,{\bf 1}) = \mI(u^*)$. Therefore there is
a $\vr \in \Hom(\Ker(v),A)$ so that $\psi = u \circ \vr$, which furnishes
$\Ker(v) \subset \mI(u)$. Consequently, the sequence
$$
A \overset{u}{\longrightarrow} B \overset{v}{\longrightarrow} C
$$
is exact.

Finally, let us show that $u$ is a strict morphism. Indeed, by taking
$D = \mI(u)$ endowed with the topology induced by that of $B$ and
$\psi\colon \mI(u) \to B$ the inclusion mapping, one has
$\psi \in \Ker(v^*,{\bf 1}) = \mI(u^*)$. Thus there is a
$\vr \in \Hom(\mI(u),A)$ so that $\psi = u \circ \vr$. Consequently,
$$
u(x) = \psi(u(x)) = u(\vr(u(x)))
$$
for all $x \in A$, and the injectivity of $u$ implies $\vr(u(x)) = x$
for all $x \in A$. Hence, if we view $u$ as a bijective morphism from
$A$ into $\mI(u)$, it follows that $u$ is a strict morphism.
\end{proof}

\begin{corollary}\label{Corollary3.2}
Let $A$ be a separable complete metrizable topological group,
$B$ a complete metrizable topological group, and $C$ an arbitrary
topological group. Let $u \in \Hom(A,B)$ and $v \in \Hom(B,C)$.
Then the following conditions are equivalent:

\begin{itemize}

\item[\rm (a)] $\mI(u)$ is closed and
  $$
  e \longrightarrow A \overset{u}{\longrightarrow} B \overset{v}{\longrightarrow} C
  $$
  is an exact sequence of group homomorphisms;

\item[\rm (b)] for each topological group $D$, the diagram
  $$
  \Hom(D,A) \overset{u^*}{\longrightarrow}\Hom(D,B)^{\overset{v^*}{\longrightarrow}}_{\underset{\bf 1}{\longrightarrow}} \Hom(D,C)
  $$
  of mappings is exact.

\end{itemize}
\end{corollary}

\begin{proof}
Since every bijective morphism from a separable complete metrizable topological group into a complete metrizable topological group is a
strict morphism \cite{Banach}, the result follows immediately from
Theorem~\ref{Theorem3.1}.
\end{proof}

If $A$, $B$ are arbitrary groups, $\Hom_a(A,B)$ will denote the set of all
group homomorphisms from $A$ into $B$.

\begin{corollary}\label{Corollary3.3}
Let $A$, $B$, $C$ be groups and let $u\colon A \to B$, $v\colon B \to C$ be
group homomorphisms. Then the following conditions are equivalent:

\begin{itemize}

\item[\rm (a)] the sequence
  $$
  e \longrightarrow A \overset{u}{\longrightarrow} B \overset{v}{\longrightarrow} C
  $$
  of group homomorphisms is exact;

\item[\rm (b)] for each group $D$, the diagram
  $$
  \Hom_a(D,A) \overset{u^+}{\longrightarrow} \Hom_a(D,B)^{\overset{v^+}{\longrightarrow}}_{\underset{\bf 1}{\longrightarrow}} \Hom_a(D,C)
  $$
  of mappings is exact, where  $u^+(\vr) = u \circ \vr$ for $\vr \in \Hom_a(D,A)$,\, $v^+(\psi) = v \circ \psi$ for $\psi \in \Hom_a(D,B)$ and ${\bf 1}(\psi) = e$ for $\psi \in \Hom_a(D,B)$.

\end{itemize}
\end{corollary}

\begin{proof}
In what follows a group $G$ endowed with the indiscrete topology will be
represented by $G_i$\,. Then
$$
\Hom_a(D,G) = \Hom(D,G_i)
$$
for every topological group $D$; in particular,
$$
\Hom_a(D,G) = \Hom(D_i,G_i)
$$
for every group $D$.

The topological groups $A_i, B_i$ and $C_i$ will be considered in the proof.

\smallskip

\noindent (a) $\Rightarrow$ (b):\, It is obvious that $u\colon A_i \to B_i$ is a strict morphism. Therefore, by
(a) $\Rightarrow$ (b) of Theorem~\ref{Theorem3.1}, the diagram
\begin{align*}
\Hom_a(D,A) = \Hom(D_i,A_i) \overset{u^+}{\longrightarrow} &\Hom_a(D,B) = \Hom(D_i,B_i)\\
{}^{\overset{v^+}{\longrightarrow}}_{\underset{\bf 1}{\longrightarrow}} &\Hom_a(D,C) = \Hom(D_i,C_i)
\end{align*}
is exact for every group $D$.

\smallskip

\noindent (b) $\Rightarrow$ (a):\, Let $D$ be an arbitrary topological group.
By hypothesis, the diagram
\begin{align*}
\Hom_a(D,A) = \Hom(D,A_i) \overset{u^*}{\longrightarrow} &\Hom_a(D,B) = \Hom(D,B_i)\\
{}^{\overset{v^*}{\longrightarrow}}_{\underset{\bf 1}{\longrightarrow}} &\Hom_a(D,C) = \Hom(D,C_i)
\end{align*}
is exact. Therefore, by (b) $\Rightarrow$ (a) of Theorem~\ref{Theorem3.1},
the sequence
$$
e \longrightarrow A \overset{u}{\longrightarrow} B \overset{v}{\longrightarrow} C
$$
is exact.
\end{proof}

Let $A$ be the additive group of real numbers endowed with the discrete
topology and $B$ the additive group of real numbers endowed with the usual
topology. Let $u \in \Hom(A,A)$ be given by $u(x) = e$ and let
$v \in \Hom(A,B)$ be given by $v(x) = x$. It is obvious that the sequence
$$
A \overset{u}{\longrightarrow} A \overset{v}{\longrightarrow} B \longrightarrow e
$$
of group homomorphisms is exact and that $v$ is not a strict morphism.
Nevertheless, the sequence
$$
\Hom(B,A) \overset{v_*}{\longrightarrow} \Hom(A,A)\overset{u_*}{\longrightarrow} \Hom(A,A)
$$
of group homomorphisms is not exact. For, if it were, the mapping
$x \in B \mapsto x \in A$ would be continuous, which is not the case.

The above-mentioned example shows that the strictness of $v$ is essential
for the validity of the implication (a) $\Rightarrow$ (b) in the following

\begin{theorem}\label{Theorem3.4}
Let $A$, $B$, $C$ be topological groups, $u \in \Hom(A,B)$ and
$v \in \Hom(B,C)$. Then the following conditions are equivalent:

\begin{itemize}

\item[\rm (a)] $v$ is a strict morphism and the sequence
  $$
  A \overset{u}{\longrightarrow} B \overset{v}{\longrightarrow} C \longrightarrow e
  $$
  of group homomorphisms is exact.

\item[\rm (b)] $\mI(u)$ is a normal subgroup of $B$,\, $\mI(v)$ is a normal
  subgroup of $C$ and, for each topological group $D$, the diagram
  $$
  \Hom(C,D)\overset{v_*}{\longrightarrow}\Hom(B,D)^{\overset{u_*}{\longrightarrow}}_{\underset{\bf 1}{\longrightarrow}} \Hom(A,D)
  $$
  of mappings is exact.

\end{itemize}
\end{theorem}

\begin{proof}
(a) $\Rightarrow$ (b):\, First of all, to say that the sequence
$$
A \overset{u}{\longrightarrow} B \overset{v}{\longrightarrow} C \longrightarrow e
$$
is exact is equivalent to saying that $\mI(u) = \Ker(v)$ and $v$ is
surjective; hence, in this case, $\mI(u)$ is a normal subgroup of $B$
and $\mI(v)$\,  is a normal subgroup of $C$.

Let $D$ be an arbitrary topological group, and let us show the exactness
of the diagram
$$
\Hom(C,D) \overset{v_*}{\longrightarrow} \Hom(B,D)^{\overset{u_*}{\longrightarrow}}_{\underset{\bf 1}{\longrightarrow}} \Hom(A,D).
$$

The injectivity of $v_*$ follows from the surjectivity of $v$.
Now let us prove that $\mI(v_*) = \Ker(u_*,{\bf 1})$.
Indeed, since $\mI(u) \subset \Ker(v)$, we get
$\mI(v_*) \subset \Ker(u_*,\bf 1)$. On the other hand, let
$w \in \Ker(u_*,\bf 1)$. Hence $w \circ u = e$\,, which implies
$$
\Ker(v) = \mI(u) \subset \Ker(w).
$$
Put $w'(v(y)) = w(y)$ for $y \in B$; $w'$ is well-defined and is a group
homomorphism from $C$ into $D$. Moreover, $w' \in \Hom(C,D)$ because $v$
is a strict morphism and $w$ is continuous. Finally,
$$
w = w' \circ v = v_*(w') \in \mI(v_*),
$$
and $\Ker(u_*,{\bf 1}) \subset \mI(v_*)$. Therefore
$\mI(v_*) = \Ker(u_*,\bf 1)$.

\smallskip

\noindent (b) $\Rightarrow$ (a):\, By hypothesis, the diagram
$$
\Hom(C,D) \overset{v_*}{\longrightarrow} \Hom(B,D){}^{\overset{u_*}{\longrightarrow}}_{\underset{\bf 1}{\longrightarrow}} \Hom(A,D)
$$
is exact for each topological group $D$. Then, by taking $D$ as the quotient
topological group $C/\mI(v)$, the injectivity of the mapping
$$
\Hom(C,C/\mI(v)) \overset{v_*}{\longrightarrow} \Hom(B, C/\mI(v))
$$
implies the exactness of the sequence
$$
B \overset{v}{\longrightarrow} C \longrightarrow e\,,
$$
that is, the surjectivity of $v$, because the canonical surjection
$\pi\colon C \to C/\mI(v)$ satisfies $v_*(\pi) = e$\,. Now, let us prove
the exactness of the sequence
$$
A \overset{u}{\longrightarrow} B \overset{v}{\longrightarrow} C.
$$
In fact, by taking $D = C$ we get
$$
e = (u_* \circ v_*)(1_C) = v \circ u,
$$
and hence $\mI(u) \subset \Ker(v)$. On the other hand, by taking $D$ as the
quotient topological group $B/\mI(u)$\,, we have that the canonical
surjection $\xi\colon B \to B/\mI(u)$ belongs to $\Ker(u_*,{\bf 1})$.
Thus there is a $\psi \in \Hom(C, B/\mI(u))$ such that
$\xi = v_*(\psi) = \psi \circ v$. Consequently,
$$
\xi(y) = \psi(v(y)) = \psi(e) = e
$$
for all $y \in \Ker(v)$, that is, $\Ker(v) \subset \Ker(\xi) = \mI(u)$.
Therefore $\mI(u) = \Ker(v)$. Finally, $\psi(v(y)) = \xi(y) = y\Ker(v)$
for all $y \in B$, that is, $\psi$ is the inverse of the bijective morphism
$y\Ker(v) \in B/\Ker(v) \mapsto v(y) \in C$; hence $v$ is a strict morphism.
\end{proof}

\begin{example}\label{Example3.5}
Let $E$ be an arbitrary Hausdorff locally convex space over $\mathbb{K}$
($\mathbb R$ or $\mathbb C$) and $M$ a closed subspace of $E$. If $E'$ is
the topological dual of $E$, a straightforward argument shows that
$E' = \Hom(E,\mathbb K)$. The vector space $E/M$, endowed with the quotient
topology, is a Hausdorff locally convex space over $\mathbb K$; moreover,
the canonical injection $i\colon M \to E$ and the canonical surjection
$\pi\colon E \to E/M$ are continuous linear mappings, $\pi$ being a strict
morphism. Since the sequence
$$
M \overset{i}{\longrightarrow} E \overset{\pi}{\longrightarrow} E/M \longrightarrow e
$$
is exact, Theorem~\ref{Theorem3.4} ensures the exactness of the sequence
$$
e \longrightarrow \big(E/M\big)' \overset{\pi^t}{\longrightarrow} E' \overset{i^t}{\longrightarrow} M'
$$
of linear mappings, where $i^t$ (resp. $\pi^t$) denotes the transpose of $i$
(resp. $\pi$). Therefore $\pi^t\colon \big(E/M\big)' \to {\rm Im}(\pi^t) =
{\rm Ker}(i^t)$ is a vector space isomorphism. On the other hand,
$$
\Ker(i^t) = \{\varphi \in E'; \varphi|M = e\} = M^{\bot},
$$
where $M^\bot$ is the orthogonal of $M$ with respect to the canonical dual
pair $(E,E')$. Finally, if we endow $\big(E/M\big)'$ with the weak topology
$\sigma\big(\big(E/M\big)', E/M\big)$ and $M^\bot$ with the topology induced
by the weak topology $\sigma(E',E)$, it is not hard to see that $\pi^t$ is a
locally convex space isomorphism.
\end{example}

\begin{corollary}\label{Corollary3.6}
Let $A$ be an arbitrary topological group, $B$ a 
separable complete metrizable topological group, and $C$ a complete metrizable topological
group. Let $u \in \Hom(A,B)$ and $v \in \Hom(B,C)$. Then the following
conditions are equivalent:

\begin{itemize}

\item[\rm (a)] the sequence
  $$
  A \overset{u}{\longrightarrow} B \overset{v}{\longrightarrow} C \longrightarrow e
  $$
  of group homomorphisms is exact;

\item[\rm (b)] $\mI(u)$ is a normal subgroup of $B$,\, $\mI(v)$ is a normal
  subgroup of $C$ and, for each topological group $D$, the diagram
  $$
  \Hom(C,D) \overset{v_*}{\longrightarrow} \Hom(B,D){}^{\overset{u_*}{\longrightarrow}}_{\underset{\bf 1}{\longrightarrow}} \Hom(A,D)
  $$
  of mappings is exact.

\end{itemize}
\end{corollary}

\begin{proof}
Analogous to that of Corollary~\ref{Corollary3.2}, by applying
Theorem~\ref{Theorem3.4} in place of Theorem~\ref{Theorem3.1}.
\end{proof}

\begin{corollary}\label{Corollary3.7}
Let $A$, $B$, $C$ be groups, and let $u\colon A \to B$ and $v\colon B \to C$
be group homomorphisms. Then the following conditions are equivalent:

\begin{itemize}

\item[\rm (a)] the sequence
  $$
  A \overset{u}{\longrightarrow} B \overset{v}{\longrightarrow} C \longrightarrow e
  $$
  of group homomorphisms is exact;

\item[\rm (b)] $\mI(u)$ is a normal subgroup of $B$,\, $\mI(v)$ is a normal
  subgroup of $C$ and, for each group $D$, the diagram
  $$
  \Hom_a(C,D) \overset{v_+}{\longrightarrow} \Hom_a(B,D){}^{\overset{u_+}{\longrightarrow}}_{\underset{\bf 1}{\longrightarrow}} \Hom_a(A,D)
  $$
  of mappings is exact, where  $v_+(\vr) = \vr \circ v$ for
  $\vr \in \Hom_a(C,D)$, $u_+(\psi) = \psi\circ u$ for $\psi \in \Hom_a(B,D)$
  and ${\bf 1}(\psi)=e$ for $\psi \in \Hom_a(B,D)$.

\end{itemize}
\end{corollary}

\begin{proof}
In what follows a group $G$ endowed with the discrete topology will be
represented by $G_d$\,. Then
$$
\Hom_a(G,D) = \Hom(G_d,D)
$$
for every topological group $D$; in particular,
$$
\Hom_a(G,D) = \Hom(G_d,D_d)
$$
for every group $D$.

The topological groups $A_d, B_d$ and $C_d$ will be considered in the proof.

\smallskip

\noindent (a) $\Rightarrow$ (b):\, It is obvious that  $v\colon B_d \to C_d$ is a strict morphism. Therefore, by
(a) $\Rightarrow$ (b) of Theorem~\ref{Theorem3.4}, $\mI(u)$ (resp. $\mI(v)$)
is a normal of $B$ (resp. $C)$ and the diagram
\begin{align*}
\Hom_a(A,D) = \Hom(A_d,D_d) \overset{v_+}{\longrightarrow} &\Hom_a(B,D) = \Hom(B_d,D_d)\\
{}^{\overset{u_+}{\longrightarrow}}_{\underset{\bf 1}{\longrightarrow}} &\Hom_a(C,D) = \Hom(C_d,D_d)
\end{align*}
is exact for every group $D$.

\smallskip

\noindent (b) $\Rightarrow$ (a):\, Let $D$ be an arbitrary topological group.
By hypothesis, the diagram
\begin{align*}
\Hom_a(C,D) = \Hom(C_d,D) \overset{v_*}{\longrightarrow} &\Hom_a(B,D) = \Hom(B_d,D)\\
{}^{\overset{u_*}{\longrightarrow}}_{\underset{\bf 1}{\longrightarrow}} &\Hom_a(A,D) = \Hom(A_d,D)
\end{align*}
is exact. Therefore, by (b) $\Rightarrow$ (a) of Theorem~\ref{Theorem3.4},
the sequence
$$
A \overset{u}{\longrightarrow} B \overset{v}{\longrightarrow} C \longrightarrow e
$$
is exact.
\end{proof}


\end{document}